\documentclass[12pt]{amsart}
\usepackage{amsmath}
\usepackage{hyperref}
\usepackage{amssymb}
\usepackage{latexsym}
\usepackage{amscd}
\usepackage{graphicx} 
\usepackage{amsthm}
\usepackage{mathrsfs}
\usepackage{xypic}
\usepackage{bm}

\newdimen\AAdi%
\newbox\AAbo%
\def\AAk#1#2{\s_etbox\AAbo=\hbox{#2}\AAdi=\wd\AAbo\kern#1\AAdi{}}%
\def\AAr#1#2#3{\s_etbox\AAbo=\hbox{#2}\AAdi=\ht\AAbo\raise#1\AAdi\hbox{#3}}%
\font\tenmsb=msbm10 at 12pt \font\sevenmsb=msbm7 at 8pt
\font\fivemsb=msbm5 at 6pt
\newfam\msbfam
\textfont\msbfam=\tenmsb \scriptfont\msbfam=\sevenmsb
\scriptscriptfont\msbfam=\fivemsb
\def\Bbb#1{{\tenmsb\fam\msbfam#1}}
\newtheorem{thm}{Theorem}[section]

\newtheorem{cor}{Corollary}[section]
\newtheorem{rem}{Remark}[section]
\newtheorem{pro}{Proposition}[section]

\newcommand{\ba}{\begin{array}}
\newcommand{\ea}{\end{array}}

\newcommand{\Section}[2]{\setcounter{equation}{0}
\allowdisplaybreaks
\section[#1]{#2}}

\def\n{\nabla}
\def\bn{\overline\nabla}
\def\ir#1{\mathbb R^{#1}}

\def\f#1#2{\frac{#1}{#2}}

\def\grs#1#2{\bold G_{#1,#2}}

\def\dt#1{\frac {d\,#1}{d\,t}}
\def\mc#1{\mathcal{#1}}

\def\td{\tilde}

\def\a{\alpha}
\def\be{\beta}

\def\p#1{\partial #1}

\def\de{\delta}
\def\De{\Delta}

\def\ep{\epsilon}
\def\G{\Gamma}
\def\g{\gamma}

\def\la{\lambda}

\def\om{\omega}
\def\Om{\Omega}
\def\th{\theta}

\def\w{\wedge}

\def\ze{\zeta}

\def\Hess{\mbox{Hess}}
\def\R{\Bbb{R}}

\def\lan{\langle}
\def\ran{\rangle}
\def\ra{\rightarrow}

\def\bn{\bar{\nabla}}

\begin{document}

\title[The rigidity theorems of self shrinkers]{The Rigidity Theorems of Self Shrinkers \\ via Gauss maps}
\author{Qi Ding}
\author{Y.L. Xin}
\author{Ling Yang}
\address{Institute of Mathematics, Fudan University,
Shanghai 200433, China} \email{09110180013@fudan.edu.cn}
\email{ylxin@fudan.edu.cn}\email{yanglingfd@fudan.edu.cn}
\thanks{The research was partially supported by
NSFC}

\begin{abstract}
We study the rigidity results for self-shrinkers in Euclidean space
by restriction of the image under the Gauss map. The geometric
properties of the target manifolds carry into effect. In the
self-shrinking hypersurface situation  Theorem \ref{Ri1} and Theorem
\ref{Ri2} not only improve the previous results, but also are
optimal. In higher codimensional case, using geometric properties of
the Grassmanian manifolds (the target manifolds of the Gauss map) we
give a rigidity theorem for self-shrinking graphs.

\end{abstract}

\maketitle

\Section{Introduction}{Introduction}

\medskip

Minimal submanifolds and self-shrinkers both are special solutions
to the mean curvature flow. Those two subjects share many geometric
properties, as shown in \cite{CM1}. We continue to study rigidity
properties of self-shrinkers. In the previous work we discuss the
gap phenomena for squared norm of the second fundamental form for
self-shrinkers \cite{DX1}. For  submanifolds in Euclidean space we
have the important Gauss map, which plays essential role in
submanifold theory. In the present paper we shall study the gap
phenomena of the image under the Gauss maps for self-shrinkers. In
the literature \cite{EH}\cite{CM1} the polynomial volume growth is
an adequate assumption for the complete non-compact self-shrinkers.
Ding-Xin \cite{DX} showed that the properness shall guarantee  the
Euclidean volume growth. Afterwards, Chen-Zhou \cite{CZ} proved that
the inverse is also true. It is unclear if there exists a complete
improper self-shrinker in Euclidean space. Now, we only study
properly immersed self-shrinkers. We pursue the results that a
complete properly immersed self-shrinker would become an affine
linear subspace or a cylinder, if its Gauss image is sufficiently
restricted.

In the next section we show that the Gauss map of a self-shrinker is
a weighted harmonic map, which is a conclusion of the Ruh-Vilms type
result for self-shrinkers, see Theorem \ref{RV}. We also derive a
composition formula for the drift-Laplacian operator defined on
self-shrinkers, which enables us to obtain some results of
self-shrinkers via properties of the target manifold of the Gauss
map in the subsequent sections of this paper.

In \S 3 we study the codimension one case. If $M$ is an entire
graphic self shrinking hypersurface in $\R^{n+1},$ Ecker-Huisken
showed that  $M$ is a hyperplane \cite{EH} under the assumption of
polynomial volume growth, which was removed  by Wang \cite{W}.
Namely, any entire graphic self-shrinking hypersurface in Euclidean
space has to be a hyperplane. It is in sharp contrast to the case of
minimal graphic hypersurfaces. For  constant mean curvature surfaces
in $\ir{3}$ there is the well-known results, due to
Hoffeman-Osserman-Schoen \cite{HOS}. Their results show that a plane
or a circular cylinder could be characterized by its Gauss image
among other complete constant mean curvature surfaces in $\ir{3}.$
In this circumstance we consider a properly immersed self-shrinking
hypersurface $M$ in $\R^{n+1}$. Now the target manifolds of the
Gauss map for a self shrinking hypersurface in $\R^{n+1}$  is the
unit sphere. We  obtain a counterpart of their results and prove
that if the image under the Gauss map is contained in an open
hemisphere (which includes the case of graphic self shrinking
hypersurfaces in $\R^{n+1}$), then $M$ is a hyperplane. If the image
under the Gauss map is contained in a closed hemisphere, then $M$ is
a hyperplane or a cylinder over a self-shrinker of one dimension
lower, see Theorem \ref{Ri1}. The convex geometry of the sphere has
been studied extensively by Jost-Xin-Yang \cite{J-X-Y1}. Using their
technique we could improve the first part of  Theorem \ref{Ri1} and
obtain Theorem \ref{Ri2}, which is the best possible. The omitting
range of the Gauss image would be the codimension one closed
hemisphere $\overline{S}_+^{n-1},$ much smaller than the closed
hemisphere $\overline{S}_+^n$ in Theorem \ref{Ri1}.

In \S 4 we study the higher codimensional graphic situation. The
target manifold of the Gauss map is the Grassmannian manifold now.
To study the higher codimensional Bernstein problem, Jost-Xin-Yang
obtained some interesting geometric properties of the Grassmannian
manifolds and developed some skilled technique \cite{J-X-Y}. This
enables us to obtain rigidity results of higher codimension for
self-shrinkers.  Using Theorem 3.1 in \cite{J-X-Y}, Ding-Wang
obtained a result for this problem \cite{DW}. Now, using the method
of Theorem 3.1 in \cite {J-X-Y}  we prove Proposition \ref{subhar3}
to fit the present situation. Therefore, we obtain Theorem
\ref{Ri3}, which improves corresponding results in \cite{DW}. As
for   Lagrangian self-shrinkers (a special higher codimensional case),
readers are referred to the papers \cite{CCY},\cite{HW} and
\cite{DX2}.

The weighted harmonic maps have already been introduced and studied.
For convenience we describe its basic notion in an appendix, as the
final section.

\bigskip

\Section{Gauss maps for self shrinkers}{Gauss maps for
self shrinkers}

\medskip

If $M$ is an oriented submanifold in $\R^{m+n}$, we can define the
Gauss map $\g: M\ra \grs{n}{m}$ that is obtained by parallel
translation of $T_pM$ to the origin in the ambient space $\ir{m+n}.$
Here
 $\grs{n}{m}$ is the Grassmannian manifolds constituting of all oriented $n$-subspaces in
 $\R^{m+n}$. It is a Riemannian symmetric space of compact type. When $m=1$, $\grs{n}{1}$ becomes
Euclidean sphere. The properties of the Gauss map implies the
properties of the submanifolds.

Let $(M,g)$ be an $n$-dimensional Riemannian manifold, and $X: M
\rightarrow\ir{m+n}$ be an isometric immersion. Let $\n$ and
$\bn$ be Levi-Civita connections on $M$ and $\R^{m+n}$,
respectively. The second fundamental form $B$ is defined by
$B_{VW}=(\bn_VW)^N=\bn_VW-\n_VW$ for any vector fields $V,W$
along the submanifold $M$, where $(\cdots)^N$ is the projection
onto the normal bundle $NM$. Similarly, $(\cdots)^T$ stands for
the tangential projection.  Taking the trace of $B$ gives the
mean curvature vector $H$ of $M$ in $\ir{m+n}$,  a cross-section
of the normal bundle. In what follows we use $\n$ for natural
connections on various bundles for notational simplicity if there
is no ambiguity from the context. For $\nu\in\G(NM)$ the shape
operator $A^\nu: TM\to TM$, defined by $A^\nu(V)=-(\bn_V\nu)^T$,
satisfies $\left<B_{V W}, \nu\right>=\left<A^\nu(V), W\right>.$

The second fundamental form, curvature tensors of the submanifold,
curvature tensor of the normal bundle and that of the ambient
manifold satisfy the Gauss equations, the Codazzi equations and the
Ricci equations (see \cite{X2}, for example).

We now consider the mean curvature flow for a submanifold  $M$ in
$\ir{m+n}.$ Namely, consider a one-parameter family $X_t=X(\cdot,
t)$ of immersions $X_t:M\to \ir{m+n}$ with corresponding images
$M_t=X_t(M)$ such that
\begin{equation*}\left\{\begin{split}
\dt{}X(p, t)&=H(p, t),\qquad p\in M\\
X(p, 0)&=X(p)
\end{split}\right.
\end{equation*}
is satisfied, where $H(p, t)$ is the mean curvature vector of $M_t$
at $X(p, t)$ in $\ir{m+n}.$

An important class of solutions to the above mean curvature flow
equations are self-similar shrinkers, whose profiles,
self-shrinkers, satisfy a system of quasi-linear elliptic PDE of the
second order
\begin{equation}\label{SS}
H = -\frac{X^N}{2}.
\end{equation}

Let $\De$, $\mathrm{div}$ and $d\mu$ be Laplacian, divergence and
volume element on $M$ induced by the metric $g$, respectively.
Colding and Minicozzi in \cite{CM1} introduced a linear operator,
drift-Laplacian
\begin{equation}
\mc{L}=\De-\frac{1}{2}\langle
X,\n(\cdot)\rangle=e^{\f{|X|^2}4}\mathrm{div}(e^{-\f{|X|^2}4}\n(\cdot))
\end{equation}
on self-shrinkers. They showed that $\mc{L}$ is self-adjoint
respect to the measure $e^{-\f{|X|^2}4}d\mu$. In the present paper we
carry out integrations with respect to this measure. We denote
\begin{equation}\label{rho}
\rho:=e^{-\f{|X|^2}4}
\end{equation}
 and the volume form
 $d\mu$ might be omitted in the integrations for notational simplicity.

Especially if $M$ is a graph over a domain $\Om\in\R^{n}$, namely,
$$M=\big\{(x_1,\cdots,x_n,u^1,\cdots,u^m):  u^{\alpha}=u^{\alpha}(x_1,\cdots,x_n)\big\}.$$
Then the induced metric of $M$
$$g=g_{ij}dx_idx_j=(\de_{ij}+u^\a_iu^\a_j)dx_idx_j$$
with $u_i^\a=\f{\p u^\a}{\p x_i}$.
Let
$x:=(x_1,x_2,\cdots,x_n)$, $(g^{ij})$ be the inverse matrix of
$(g_{ij})$ and
\begin{equation}\label{v}
v=\De_u=\det g
\end{equation}
be the slope of the vector-valued function $u$.
Then the equation \eqref{SS} can be written as
the  following elliptic system(see \cite{DW})
\begin{equation}\label{GSS}
\sum_{i,j=1}^ng^{ij}u^{\alpha}_{ij}=\f12(x\cdot D
u^{\alpha}-u^{\alpha}),
\end{equation}
where $u_{ij}^\a:=\f{\p^2 u^\a}{\p x_i\p x_j}$ and $Du^\a:=(u^\a_1,\cdots,u^\a_n)$.

By a straightforward calculation,
\begin{equation}\aligned\label{vgij1}
\p_i(vg^{ij})=&\f12vg^{kl}\p_ig_{kl}g^{ij}-vg^{ki}\p_ig_{kl}g^{lj}\\
=&\f12vg^{kl}(u_{ki}^\a u_l^\a+u_{k}^\a u_{li}^\a)g^{ij}-vg^{ki}(u_{ki}^\a u_l^\a+u_{k}^\a u_{li}^\a)g^{lj}\\
=&-vg^{ki}u_{ki}^\a u_l^\a g^{lj}.
\endaligned
\end{equation}
Substituting \eqref{GSS} into \eqref{vgij1} gives
\begin{equation}\aligned\label{vgij2}
\p_i(vg^{ij})=&\f12v(u^\a-x_iu_i^\a)u_l^\a g^{lj}\\
=&\f12vu^\a u_l^\a g^{lj}-\f12vx_i(g_{il}-\de_{il})g^{lj}\\
=&\f12vu^\a u_i^\a g^{ij}-\f12vx_j+\f12vx_ig^{ij}.
\endaligned
\end{equation}
For any $C^2-$function $f$ in $M$, combining \eqref{vgij2}, we have
\begin{equation}\aligned
\mathcal{L}f=&\f1v e^{\f{|X|^2}4}\f{\p}{\p x_i}\left(g^{ij}ve^{-\f{|X|^2}4}\f{\p}{\p x_j}f\right)\\
=&g^{ij}f_{ij}+\f1v \p_i(g^{ij}v)f_j-\f12g^{ij}(x_i+u^\a u^\a_i)f_j\\
=&g^{ij}f_{ij}+\left(\f12u^\a u_i^\a g^{ij}-\f12x_j+\f12x_ig^{ij}\right)f_j-\f12g^{ij}(x_i+u^\a u^\a_i)f_j\\
=&g^{ij}f_{ij}-\f12x_jf_j.
\endaligned
\end{equation}

\begin{rem} For graphic self shrinkers in $\R^{m+n}$, the operator $L_g$ defined in \cite{DW} is
precisely the drift-Laplace $\mathcal{L}$. Please see \cite{DX2} for Lagrangian case.
\end{rem}

Once $M$ is  minimal, the Gauss map of $M$ must be a harmonic map.
This is a conclusion of the well-known Ruh-Vilms theorem \cite{RV},
which reveals the close relationship between Liouville type theorems
for harmonic maps and Bernstein type results for minimal
submanifolds. There is also a notion of $\rho-$ weighted harmonic
maps. Its definition  shall be given in Appendix. We have the
counterpart  of the Ruh-Vilms theorem.

\begin{thm}\label{RV}
For an oriented $n-$dimensional submanifold $X:M\to\ir{m+n}$ its
Gauss map $\g: M\to \grs{n}{m}$ is $\rho-$weighted harmonic map if
and only if $H+\f{1}{2}X^N$ is a parallel vector field in the normal
bundle $NM$.
\end{thm}

\begin{proof}

Let $\{e_1,\cdots,e_n\}$ be a local tangent orthonormal frame field
on $M$ and $\{\nu_1,\cdots,\nu_m\}$ be a local normal orthornormal
frame field on $M$, and we assume $\n e_i=0$ and $\n \nu_\a=0$ at
the considered point. Here and in the sequel we use summation
convention and assume the range of indices.
$$1\leq i,j,k\leq n,\qquad 1\leq \a,\be\leq m.$$
Using Pl\"ucker coordinates, the Gauss map $\g$ could be described as
$\g(p)=e_1\w\cdots\w e_n$, thus
\begin{equation}\aligned\label{tension1}
d\g(e_i)&=\bar{\n}_{e_i}(e_1\w\cdots\w e_n)\\
        &=\sum_j e_1\w\cdots\w B_{e_ie_j}\w \cdots\w e_n\\
        &=\sum_j e_1\w\cdots\w h_{\a,ij} \nu_\a\w \cdots\w e_n\\
        &=h_{\a,ij} e_{\a j}
        \endaligned
        \end{equation}
where $\{h_{\a,ij}=\lan B_{e_ie_j},\nu_\a\ran:1\leq i,j\leq n,1\leq
\a\leq m\}$ are coefficients of the second fundamental form, and
$e_{\a j}$ is obtained by replacing $e_j$ by $\nu_\a$ in
$e_1\w\cdots\w e_n$. We note that $\{e_{\a j}:1\leq j\leq n, 1\leq
\a\leq m\}$ is an orthornormal basis of the tangent space of
$\grs{n}{m}$ at $e_1\w \cdots\w e_n$.

At the considered point,
\begin{equation}\aligned
\n_{e_i}e_{\a j}=&\n_{e_i}(e_1\w \cdots\w \nu_\a\w \cdots\w e_n)\\
                =&\sum_{k}e_1\w\cdots\w \n_{e_i}e_k\w\cdots \w \nu_\a\w\cdots\w e_n\\
                &+e_1\w \cdots\w \n_{e_i}\nu_\a\w \cdots\w e_n\\
                =&0.
                \endaligned
\end{equation}

Using Codazzi equations one can obtain
\begin{equation}\aligned\label{tension2}
\n_{e_i}h_{\a,ij}&=\n_{e_i}\lan B_{e_i e_j},\nu_\a\ran=\lan (\n_{e_i}B)_{e_ie_j},\nu_\a\ran\\
                 &=\lan (\n_{e_j}B)_{e_ie_i},\nu_\a\ran= \n_{e_j}H^\a
                 \endaligned
\end{equation}
with $H^\a:=\lan H,\nu_\a\ran$ the coefficients of the mean
curvature vector.

Combining with (\ref{tension1})-(\ref{tension2}) gives
\begin{equation}\label{tension3}\aligned
(\n_{e_i}d\g)e_i&=\n_{e_i}d\g(e_i)=(\n_{e_i}h_{\a,ij})e_{\a j}+h_{\a,ij} \n_{e_i}e_{\a j}\\
&=(\n_{e_j}H^\a) e_{\a j}.
\endaligned
\end{equation}

Since $\rho=e^{-\f{|X|^2}{4}}$,
\begin{equation}\label{tension4}\aligned
\n_{e_i}\rho&=-\f{1}{4}\rho\n_{e_i}|X|^2=-\f{1}{2}\rho\lan X,\bar{\n}_{e_i}X\ran\\
            &=-\f{1}{2}\rho\lan X,e_i\ran.
            \endaligned
\end{equation}
Let $X^\a:=\lan X,\nu_\a\ran=\lan X^N,\nu_\a\ran$, then
\begin{equation}\label{tension5}\aligned
\n_{e_j}X^\a&=\lan \bar{\n}_{e_j}X,\nu_\a\ran+\lan X,\bar{\n}_{e_j}\nu_\a\ran&=\lan e_j,\nu_\a\ran-\lan X,h_{\a,ij} e_i\ran\\
            &=-h_{\a,ij}\lan X,e_i\ran.
            \endaligned
\end{equation}

In conjunction with (\ref{tension1}), (\ref{tension3}), (\ref{tension4}) and (\ref{tension5}) we have
\begin{equation}\aligned
\tau_\rho(\g):=&\rho^{-1}\big(\n_{e_i}(\rho d\g)\big)e_i=\rho^{-1}(\n_{e_i}\rho)d\g(e_i)+\n_{e_i}(d\g)(e_i)\\
              =&\Big[-\f{1}{2}h_{\a,ij}\lan X,e_i\ran+\n_{e_j}H^\a\Big]e_{\a j}\\
              =& \n_{e_j}(H^\a+\f{1}{2}X^\a)e_{\a j}.
              \endaligned
              \end{equation}
Considering the definition of $\rho-$weighted harmonic map in the
appendix, the conclusion follows.

\end{proof}

\begin{cor}\label{wrv}
If $M$ is a self-shrinker in $\ir{m+n},$ then its Gauss map $\g:M\to
\grs{n}{m}$ is a $\rho-$weighted harmonic map.
\end{cor}

Now we assume $F$ to be a $C^2$-function on $\grs{n}{m}$, then $f=F
\circ \g$ gives a $C^2$-function on $M$. We also choose a local
orthonormal frame field $\{e_i\}$ on $M$ such that $\n e_i=0$ and
$\n \nu_\a=0$ at the considered point. A straightforward calculation
shows
\begin{equation*}\aligned
\mc{L}f&=\rho^{-1}\text{div}(\rho\n f)=\rho^{-1}\n_{e_i}\big(\rho df(e_i)\big)\\
       &=\rho^{-1}\n_{e_i}\big(\rho dF(\g_* e_i)\big)=\rho^{-1}\n_{e_i}\big(dF(\rho \g_* e_i)\big)\\
       &=\rho^{-1}(\n_{e_i}(dF))(\rho \g_* e_i)+\rho^{-1}dF\Big(\big(\n_{e_i}(\rho d\g)\big)e_i\Big)\\
       &=\text{Hess}\ F(\g_* e_i,\g_* e_i)+dF\big(\tau_\rho(\g)\big)
       \endaligned
\end{equation*}
If $M$ is a self-shrinker, by Corollary \ref{wrv}, we have the
composition formula
\begin{equation}\label{com}
\mc{L}f=\text{Hess}\ F(\g_* e_i,\g_* e_i).
\end{equation}

\begin{rem} This composition formula shall play a key role in the
proof of rigidity theorems for self-shrinkers. Certainly, the above
formula could be obtained from the usual composition formula without
the notion of $\rho-$weighted harmonic maps. In fact, an extra term
in drift-Laplacian would be canceled by the tension field term.
\end{rem}

\bigskip

\Section{Rigidity results for hypersurfaces}{Rigidity results for
hypersurfaces}\label{hyp}

\medskip

Let $(\cdot,\cdot)$ be the canonical Euclidean inner product on $\R^{n+1}$, then for any fixed
$a\in S^n\subset \R^{n+1}$, $(\cdot,a)$ is obviously a smooth function on $S^n$.

By the theory of spherical geometry, the normal geodesic $\g$
starting from $x\in S^n$ and with the initial vector $v$  ($|v|=1$
and $(x,v)=0$) has the form
$$\g(t)=\cos t\ x+\sin t\ v.$$
Then
$$(\g(t),a)=\cos t\ (x,a)+\sin t\ (v,a).$$
Differentiating twice both sides of the above equation with respect to $t$ implies
$$\Hess(\cdot,a)(v,v)=-(\cdot,a).$$
In conjunction with the formula $2\Hess\ h(v,w)=\Hess\
h(v+w,v+w)-\Hess\ h(v,v)-\Hess\ h(w,w)$, it is easy to obtain
\begin{equation}\label{hess1}
\Hess(\cdot,a)=-(\cdot,a)\ g_s
\end{equation}
with $g_s$ the canonical metric on $S^n$.

Denote
\begin{equation}
F=1-(\cdot,a),
\end{equation}
then
\begin{equation}\label{hess2}
\Hess\ F=(1-F)g_s.
\end{equation}

If $M$ is a self-shrinker in $\R^{n+1}$, we put $f=F\circ \g$, then combining the composition formula
(\ref{com}) with (\ref{hess2}) yields
\begin{equation}\label{L1}\aligned
\mc{L}f&=\Hess\ F(\g_* e_i,\g_* e_i)=(1-f)\lan \g_* e_i,\g_* e_i\ran\\
       &=(1-f)|B|^2.
       \endaligned
       \end{equation}
Note that $f<1$ ($f\leq 1$) equals to say that the Gauss image of $M$ is contained in the open (closed) hemisphere centered at
$a$.

(\ref{L1}) is equivalent to
\begin{equation}\label{L2}
(1-f)|B|^2\rho=\text{div}(\rho\n f).
\end{equation}
Let $\phi$ be a smooth function on $M$ with compact supporting set.
Multiplying $\phi^2 f$ with both sides of (\ref{L2}) and then
integrating by parts imply
$$\aligned
\int_M \phi^2 f(1-f)|B|^2\rho&=\int_M \phi^2 f\text{div}(\rho\n f)\\
                             &=\int_M \text{div}(\phi^2 f\rho\n f)-\int_M \lan \n(\phi^2 f),\n f\ran\rho\\
                             &=-\int_M \phi^2|\n f|^2\rho-2\int_M \lan f\n\phi,\phi\n f\ran\rho\\
                             &\leq -\int_M \phi^2|\n f|^2\rho+\f{1}{2}\int_M\phi^2|\n f|^2\rho+2\int_M |\n\phi|^2f^2\rho\\
                             &=-\f{1}{2}\int_M \phi^2|\n f|^2\rho+2\int_M |\n \phi|^2f^2\rho.
                             \endaligned$$
i.e
\begin{equation}\label{st1}
\int_M \phi^2f(1-f)|B|^2\rho+\f{1}{2}\int_M \phi^2|\n f|^2\rho\leq 2\int_M |\n \phi|^2 f^2\rho.
\end{equation}
The above 'generalized stability inequality' enables us to obtain
the following rigidity theorem.
\bigskip
\begin{thm}\label{Ri1}
Let $M$ be a complete self-shrinker hypersurface properly immersed
in $\ir{n+1}.$ If the image under the Gauss map is contained in an
open hemisphere, then $M$ has to be a hyperplane. If the image under
the Gauss map is contained in a closed hemisphere, then $M$  is a
hyperplane or a cylinder whose cross section is an
$(n-1)$-dimensional self-shrinker in $\R^n$.
\end{thm}

\begin{proof}
In (\ref{st1}), we put $\phi$ to be a cut-off function with $\phi\equiv 1$ on $D_R$ (the intersection of the Euclidean ball
of radius $R$ and $M$), $\phi\equiv 0$ outside $D_{2R}$ and $|\n \phi|\leq \f{c_0}{R}$ with a positive constant $c_0$.
Noting that $0\leq f\leq 1$ under the Gauss image assumptions, we have
$$\aligned
\f{1}{2}\int_{D_R}|\n f|^2\rho\leq &\f{1}{2}\int_M \phi^2|\n f|^2\rho\leq 2\int_M |\n\phi|^2f^2\rho\\
\leq &\f{2c_0^2}{R^2}\int_{D_{2R}\backslash D_R}f^2\rho\leq \f{2c_0^2}{R^2}e^{-\f{R^2}{4}}\text{Vol}(D_{2R}\backslash D_R).
\endaligned$$
Since $M$ has Euclidean volume growth by a result in \cite{DX},
letting $R\ra +\infty$ we get
$$\int_M |\n f|^2\rho=0.$$
Hence $\n f\equiv 0$ and $f\equiv \text{const}$.

If $f\equiv 0$, then the Gauss image of $M$ is a single point, which implies $M$ is a hyperplane. If
$f\equiv t_0$ with $t_0\in (0,1)$, then again using (\ref{st1}) gives
$$\aligned
t_0(1-t_0)\int_{D_R}|B|^2\rho&\leq \int_M \phi^2 f(1-f)|B|^2\rho\\
                             &\leq 2\int_M |\n \phi|^2f^2\rho\leq \f{2c_0^2}{R^2}e^{-\f{R^2}{4}}\text{Vol}(D_{2R}\backslash D_R).
                             \endaligned$$
Letting $R\ra +\infty$ forces $|B|^2\equiv 0$, thus $M$ has to be a hyperplane.

If $f\equiv 1$, then the Gauss image of $M$ is contained in a subsphere of codimension 1. Without loss of generality one can assume
$\g(M)\subset \{x=(x_1,\cdots,x_{n+1})\in S^n:x_{n+1}=0\}$. Then for any $p\in M$, $\ep_{n+1}\in T_p M$. (Here and in the sequel,
$\ep_i$ ($1\leq i\leq n+1$) is a vector in $\R^{n+1}$ whose $i$-th coordinate is 1
and other coordinates are all 0.) In other words, $Y:=\ep_{n+1}$ is a tangent vector field on $M$. Let $\xi:(-\varepsilon,\varepsilon)
\ra M$ be a curve on $M$ satisfying the following ODE system
\begin{equation}\left\{\begin{split}
\dot{\xi}(t)&=Y\big(\xi(t)\big),\\
\xi(0)&=X(p).
\end{split}\right.
\end{equation}
Then $|Y|\equiv 1$ and the completeness of $M$ implies $\xi$ can be
infinitely extended towards both ends. Noting that $Y$ can be viewed
as a vector field on $\R^{n+1}$, we put $\ze:\R\ra \R^{n+1}$ to be a
curve satisfying
\begin{equation}\left\{\begin{split}
\dot{\ze}(t)&=Y\big(\ze(t)\big),\\
\ze(0)&=X(p).
\end{split}\right.
\end{equation}
Then obviously $\ze$ is the straight line going through $X(p)$ and
being perpendicular to hyperplane $\{X\in
\R^{n+1}:(X,\ep_{n+1})=0\}$. By the uniqueness of ODE system with
the given initial conditions we have $\xi(t)=\ze(t)$. Thus $M$
constitutes of straight line orthogonal to a fixed hyperplane. More
precisely, if we denote $\td{M}=M\cap \{X\in
\R^{n+1}:(X,\ep_{n+1})=0\}$, then $\td{M}$ is obviously a
self-shrinker in $\R^n$ and $M=\td{M}\times \R$.

\end{proof}

Let $(\varphi,\th)$ be the geographic coordinate of $S^2$. More precisely, there is a covering mapping $\chi:(-\f{\pi}{2},\f{\pi}{2})
\times \R\ra S^2\backslash\{N,S\}$
$$(\varphi,\th)\mapsto (\cos\varphi\cos\th,\cos\varphi\sin\th,\sin\varphi).$$
Here $N$ and $S$ are the north pole and the south pole, $\varphi$
and $\th$ are the latitude and the longitude, respectively. Note
that each level set of $\th$ is a meridian, i.e. a half of great
circle connecting the north pole and the south pole. Although $\chi$
is not one-to-one, the restriction of $\chi$ on
$(-\f{\pi}{2},\f{\pi}{2})\times (-\pi,\pi)$ is a bijective mapping
to the open domain $\Bbb{V}$ that is obtained by deleting the
International date line from $S^2$.

The longitude function $\th$ and $\Bbb{V}$ can be generalized to higher dimensional spheres.

Let $p:\R^{n+1}\ra \R^2$
$$x=(x_1,\cdots,x_{n+1})\mapsto (x_1,x_2)$$
be a natural orthogonal projection, then $p$ maps $S^n$ onto the closed unit disk $\bar{\Bbb{D}}$.
Denote
\begin{equation}
V=\bar{\Bbb{D}}\backslash \{(a,0):-1\leq a\leq 0\}
\end{equation}
and
\begin{equation}
\Bbb{V}=p^{-1}(V)\cap S^n.
\end{equation}
Then it is easily-seen that $S^n\backslash \Bbb{V}$ is a closed
hemisphere of codimension 1. So we also write $\Bbb{V}=S^n\backslash
\overline{S}_+^{n-1}$ in the following text. It is shown in
\cite{Go} \cite{J-X-Y1} that $\Bbb{V}$ is a convex supporting subset
in $S^n$, i.e. any compact subset $K\subset \Bbb{V}$ admits a
strictly convex function on it.

Obviously there is a $(0,1]$-valued function $r$ and a $(-\pi,\pi)$-valued function $\th$
on $\Bbb{V}$, such that
\begin{equation}\label{th}
p(x)=(x_1,x_2)=(r\cos\th, r\sin\th)\qquad \text{for all }x\in \Bbb{V}.
\end{equation}

$\{x_i:1\leq i\leq n+1\}$ can be viewed as coordinate functions on $S^n$, and $x_i=(x,\ep_i)$.
By (\ref{hess1}),
\begin{equation}\label{hessx}
\Hess\ x_i=-x_i\ g_s\qquad \text{for every }1\leq i\leq n+1.
\end{equation}
From (\ref{th}), $r^2=x_1^2+x_2^2$, hence
\begin{equation}\label{hessr1}\aligned
\Hess\ r^2=&2x_1\Hess\ x_1+2x_2\Hess\ x_2+2dx_1\otimes dx_1+2dx_2\otimes dx_2\\
          =&-2x_1^2\ g_s-2x_2^2\ g_s+2(\cos\th\ dr-r\sin\th\ d\th)\otimes(\cos\th\ dr-r\sin\th\ d\th)\\
           &+2(\sin\th\ dr+r\cos\th\ d\th)\otimes (\sin\th\ dr+r\cos\th\ d\th)\\
          =&-2r^2\ g_s+2dr\otimes dr+2r^2d\th\otimes d\th.
          \endaligned
\end{equation}
On the other hand,
\begin{equation}\label{hessr2}
\Hess\ r^2=2r\Hess\ r+2dr\otimes dr.
\end{equation}
(\ref{hessr1}) and (\ref{hessr2}) implies
\begin{equation}\label{hessr}
\Hess\ r=-r\ g_s+rd\th\otimes d\th.
\end{equation}
Furthermore (\ref{hessx}), (\ref{th}) and (\ref{hessr}) yield
$$\aligned
-x_1\ g_s&=\Hess\ x_1\\
       &=\cos\th\ \Hess\ r-r\sin\th\ \Hess\ \th-r\cos\th\ d\th\otimes d\th-\sin\th(dr\otimes d\th+d\th\otimes dr)\\
       &=-x_1\ g_s+x_1\ d\th\otimes d\th-r\sin\th\ \Hess\ \th-x_1\ d\th\otimes d\th-\sin\th(dr\otimes d\th+d\th\otimes dr)\\
       &=-x_1\ g_s-r\sin\th\ \Hess\ \th-\sin\th(dr\otimes d\th+d\th\otimes dr).
\endaligned$$
i.e.
$$r\sin\th\ \Hess\ \th=-\sin\th(dr\otimes d\th+d\th\otimes dr).$$
Similarly computing $\Hess\ x_2$ with the aid of (\ref{th}) and (\ref{hessr}) gives
$$r\cos\th\ \Hess\ \th=-\cos\th(dr\otimes d\th+d\th\otimes dr).$$
Therefore
\begin{equation}
\Hess\ \th=-r^{-1}(dr\otimes d\th+d\th\otimes dr).
\end{equation}
It tells us $\Hess\ \th(v,v)=0$ for any vector $v$ on $\Bbb{V}$
satisfying $\th(v)=0$; i.e. the level sets of $\th$ are all totally
geodesic hypersurfaces in $S^n$. In fact, $\Bbb{V}=S^n\backslash
\overline{S}_+^{n-1}$ has so-called warped product structure, see
\cite{So}.

If the Gauss image of $M$ is contained in $\Bbb{V}$, using
composition formula we obtain
$$\aligned\mc{L}(\th\circ \g)&=\Hess\ \th(\g_*e_i,\g_*e_i)\\
       &=-(r\circ \g)^{-1}(dr\otimes d\th+d\th\otimes dr)(\g_*e_i,\g_* e_i)\\
       &=-2(r\circ \g)^{-1}\lan \n(r\circ \g),\n(\th\circ\g)\ran.
       \endaligned$$
In the following text, $\th\circ \g$ and $r\circ \g$ are also denoted by $\th$ and $r$,
if there is no ambiguity from the context. Then the above equality can be rewritten as
\begin{equation}
\mc{L}(\th)=-2r^{-1}\lan \n r,\n \th\ran.
\end{equation}
i.e.
$$\rho^{-1}\text{div}(\rho\n\th)=-2r^{-1}\lan \n r,\n\th\ran.$$
Hence
\begin{equation}\label{div}
\aligned
\text{div}(r^2\rho\n \th)&=r^2\text{div}(\rho\n \th)+\lan \n r^2,\rho\n\th\ran\\
&=-2\lan r\n r,\rho\n \th\ran+\lan \n r^2,\rho\n\th\ran=0
\endaligned
\end{equation}
With the aid of the above formula one can improve the rigidity result in Theorem \ref{Ri1}.

\begin{thm}\label{Ri2}
Let $M^n$ be a complete self-shrinker hypersurface properly immersed
in $\R^{n+1}.$ If the image under Gauss map is contained in
$S^n\backslash \overline{S}_{+}^{n-1}$, then $M$ has to be a
hyperplane.
\end{thm}

\begin{proof}
Let $\phi$ be a smooth function on $M$ with compact supporting set,
multiplying $\phi^2\th$ with both sides of (\ref{div}) and then
integrating by parts give
\begin{equation*}\aligned
0&=\int_M \phi^2\th\text{div}( r^2\rho\n\th)\\
&=\int_M \text{div}(\phi^2\th r^2\rho\n \th)-\int_M\lan \n(\phi^2\th),\n\th\ran r^2\rho\\
&=-\int_M \phi^2|\n \th|^2r^2\rho-2\int_M\lan \th\n\phi,\phi\n\th\ran r^2\rho\\
&\leq -\f{1}{2}\int_M \phi^2|\n \th|^2r^2\rho+2\int_M |\n\phi|^2\th^2r^2\rho
\endaligned
\end{equation*}
i.e.
\begin{equation}
\int_M \phi^2|\n \th|^2r^2\rho\leq 4\int_M |\n\phi|^2\th^2r^2\rho.
\end{equation}
Choosing $\phi$ to be a cut-off function, which satisfies
$\phi\equiv 1$ on $D_R$, $\phi\equiv 0$ outside $D_{2R}$ and $|\n
\phi|\leq \f{c_0}{R}$, then combining with $\th\in (-\pi,\pi)$ and
$r\in (0,1]$ we have
\begin{equation}\aligned
\int_{D_R}|\n \th|^2r\rho&\leq \int_M \phi^2|\n \th|^2 r^2\rho\leq 4\int_M |\n \phi|^2\th^2 r^2\rho\\
   &\leq \f{4\pi^2 c_0^2}{R^2}e^{-\f{R^2}{4}}\text{Vol}(D_{2R}\backslash D_R).
   \endaligned
   \end{equation}
By letting $R\ra +\infty$ we arrive at $|\n \th|\equiv 0$. Hence $\th\equiv \th_0\in (-\pi,\pi).$

Denote $a_0=(\cos\th_0,\sin\th_0,0,\cdots,0)$, then for arbitrary $p\in M$,
$$(\g(p),a_0)=r(p)(\cos^2\th_0+\sin^2\th_0)=r(p)>0.$$
(Note that $\g(p)=(r(p)\cos\th(p),r(p)\sin\th(p),\cdots)$.) It implies the Gauss image of $M$ is contained
in an open hemisphere centered at $a_0$, and then the final conclusion immediately follows from Theorem \ref{Ri1}.

\end{proof}

\begin{rem}
It is shown in \cite{J-X-Y1} that even we add a point to
$S^n\backslash \overline{S}_{+}^{n-1}$, it will contain a great
circle. Hence the nontrivial self-shrinker $S^1\times
\R^{n-1}\subset \R^{n+1}$ whose Gauss image is just a great circle
tells us that the Gauss image restriction in Theorem \ref{Ri2} is
optimal.
\end{rem}

\bigskip

\Section{Rigidity results in High Codimension}{Rigidity results
in High Codimension}

\medskip

Via Pl\"{u}cker embedding, Grassmannian manifold $\grs{n}{m}$ can be
viewed as a minimal submanifold in a higher dimensional Euclidean
sphere. The restriction of the Euclidean inner product on $\grs{n}{m}$ is denoted by
$w:\grs{n}{m}\times \grs{n}{m}\ra \R$
\begin{equation}w(P,Q)=\lan e_1\w\cdots\w e_n,f_1\w\cdots\w f_n\ran=\det W.
\end{equation}
Here $\{e_1,\cdots,e_n\}$ and $\{f_1,\cdots,f_n\}$ are oriented orthonormal basis of
$P$ and $Q$, respectively, and $W:=\big(\lan e_i,f_j\ran\big)$. The eigenvalues of $W^T W$ are denoted
by $\mu_1^2,\cdots,\mu_n^2$, then $\mu_i$ takes value between $0$ and $1$. We also note that $\mu_i^2$
can be expressed as
\begin{equation}
\mu_i^2=\f{1}{1+\la_i^2}
\end{equation}
with $\la_i\in [0,+\infty]$.

The Jordan angles between $P$ and $Q$ are critical values of the angle $\th$ between a nonzero vector in $P$ and its
orthogonal projection in $Q$. A direct calculation shows there are $n$ Jordan angles $\th_1,\cdots,\th_n$,
with
\begin{equation}
\th_i=\arccos \mu_i.
\end{equation}
Hence
\begin{equation}
|w|=\big(\det(W^T W)\big)^{\f{1}{2}}=\prod_{i=1}^n \mu_i=\prod_{i=1}^n \cos\th_i.
\end{equation}

Fix $P_0\in \grs{n}{m}$ spanned by $\ep_1,\cdots,\ep_n$, which are complemented by
$\ep_{n+1},\cdots,\ep_{n+m}$, such that $\{\ep_1,\cdots,\ep_{n+m}\}$ is an orthonormal basis of
$\R^{m+n}$. Define
\begin{equation}
\Bbb{U}:=\{P\in \grs{n}{m}:w(P,P_0)>0\}.
\end{equation}
Our interested quantity will be
\begin{equation}
v(\cdot,P_0):=w^{-1}(\cdot,P_0)\qquad \text{on }\Bbb{U}.
\end{equation}
Then it is easily-seen that
\begin{equation}\label{v1}
v(P,P_0)=\prod_i \sec\th_i=\prod_i \mu_i^{-1}=\prod_i \sqrt{1+\la_i^2}.
\end{equation}

In this terminology, Hess$(v(\cdot, P_0))$ has been estimated in
\cite{X-Y1}. By (3.8) in \cite{X-Y1}, we have
\begin{eqnarray}\label{He}\aligned
\Hess(v(\cdot,P_0))&=\sum_{j\neq \a}v\ \om_{j\a}^2+\sum_{1\leq j\leq p}
(1+2\la_j^2)v\ \om_{jj}^2
+\sum_{1\leq j,k\leq p, j\neq k} \la_j\la_k v(\om_{jj}\otimes \om_{kk}+\om_{jk}\otimes\om_{kj})\\
&=\sum_{\max\{j,\a\}>p}v\
\om_{j\a}^2+\sum_{1\leq j\leq p}(1+2\la_j^2)v\ \om_{jj}^2
                                          +\sum_{1\leq j,k\leq p,j\neq k}\la_j\la_k v\ \om_{jj}\otimes\om_{kk}\\
&\qquad\qquad+\sum_{1\leq j<k\leq p}\Big[(1+\la_j\la_k)v\Big(\f{\sqrt{2}}{2}(\om_{jk}
+\om_{kj})\Big)^2\\
&\hskip2in+(1-\la_j\la_k)v\Big(\f{\sqrt{2}}{2}(\om_{jk}-\om_{kj})\Big)^2\Big]
\endaligned
\end{eqnarray}
with $p:=\min\{m,n\}$ and $\{\om_{i\a}:1\leq i\leq n,1\leq \a\leq
m\}$ is a dual basis of $\{e_{\a i}:1\leq i\leq n,1\leq \a\leq m\},$
namely, $\{\om_{i\a}:1\leq i\leq n,1\leq \a\leq m\}$ is a local
orthonormal coframe field on $\grs{n}{m}$ at $P=e_1\w\cdots\w e_n.$

We also have from (3.9) in \cite{X-Y1} that

\begin{equation}\label{dv}
dv(\cdot, P_0) =\sum_{1\leq j\leq p}\la_j\,v(\cdot, P_0)\om_{jj}
\end{equation}
i.e.
\begin{equation}\label{dv2}
d\log v(\cdot,P_0)=\sum_{1\leq j\leq p}\la_j \om_{jj}.
\end{equation}
Combining with (\ref{He}) and (\ref{dv2}) gives
\begin{equation}
\Hess \log v(\cdot,P_0)=g+\sum_{1\leq j\leq p}\la_j^2\om_{jj}^2
+\sum_{1\leq j,k\leq p,j\neq k}\la_j\la_k\om_{jk}\otimes \om_{kj},
\end{equation}
where $g$ is the metric tensor on $\grs{n}{m}.$

Let
\begin{equation}\label{v2}
v:=v(\cdot,P_0)\circ \g.
\end{equation}
 By (\ref{tension1}),
$$\om_{jk}(\g_* e_i)=\om_{jk}(h_{\a,il}e_{\a l})=h_{k,ij}.$$
Hence
\begin{equation}\aligned\label{dw3}
|\n \log v|^2&=\sum_i \Big[d\log v(\cdot,P_0)(\g_* e_i)\Big]^2\\
&=\sum_i \Big[\sum_j  \la_j\om_{jj}(\g_* e_i)\Big]^2\\
&=\sum_i\Big(\sum_j \la_j h_{j,ij}\Big)^2.
\endaligned
\end{equation}
Using composition formula (\ref{com}) yields
\begin{equation}\label{La3}\aligned
\mc{L}(\log v)&=\Hess \log v(\cdot,P_0)(\g_* e_i,\g_* e_i)\\
            &=|B|^2+\sum_{i, 1\leq j\leq p}\la_j^2(\om_{jj}(\g_*e_i))^2
            +\sum_{i, 1\leq j,k\leq p,j\neq k}\la_j\la_k \om_{jk}(\g_* e_i)\om_{kj}(\g_* e_i)\\
            &=|B|^2+\sum_{i, 1\leq j\leq p}\la_j^2h_{j,ij}^2+\sum_{i,1\leq j,k\leq p,j\neq k}\la_j\la_k h_{k,ij}h_{j,ik}.
            \endaligned
\end{equation}

\begin{rem}
As shown in \cite{X-Y1}, if $M$ is a graph generated by a vector-valued function $u:\Om\subset \R^n\ra \R^m$, then the function $v$ defined in
(\ref{v2}) is just the slope of $u$. Hence the two definitions of $v$ given in (\ref{v}) and (\ref{v2}) are equivalent.
\end{rem}

\begin{pro}\label{subhar3}
There exists a positive constant $C_1$. If $M$ is a self-shrinker in $\R^{n+m}$ and $v<3$ on $M$,
then
\begin{equation}\label{La4}
\mc{L} (\log v)+C_1|\n \log v|^2\geq \f{1}{2}(3-v)|B|^2.
\end{equation}
\end{pro}

\begin{proof}

It suffices to prove (\ref{La4}) at any point where $v>1.$ For any
positive constant $C_1>0$, (\ref{La3}) and (\ref{dw3}) yield
\begin{equation}\label{group}\aligned
&\mc{L}(\log v)+C_1|\n \log v|^2\\
=&|B|^2+\sum_{i,j}\la_j^2 h_{j,ij}^2+2\sum_i\sum_{j<k}\la_j\la_k h_{k,ij}h_{j,ik}+C_1\sum_i\big(\sum_j \la_j h_{j,ij}\big)^2\\
=&\sum_{\a}\sum_{i,j>p}h_{\a,ij}^2+\sum_{i>p}I_i+\sum_{i>p}\sum_{1\leq j<k\leq p}II_{ijk}
+\sum_{1<i<j<k\leq p}III_{ijk}+\sum_{1\leq i\leq p}IV_i
 \endaligned
 \end{equation}
with
\begin{eqnarray}
I_i&=&\sum_{1\leq j\leq p}(2+\la_j^2)h_{j,ij}^2+C_1\big(\sum_j \la_j h_{j,ij}\big)^2\\
II_{ijk}&=&2h_{k,ij}^2+2h_{j,ik}^2+2\la_j\la_k h_{k,ij}h_{j,ik}
\end{eqnarray}
\begin{equation}\aligned
III_{ijk}=&2h_{i,jk}^2+2h_{j,ki}^2+2h_{k,ij}^2\\
          &+2\la_i\la_j h_{i,jk}h_{j,ki}+2\la_j\la_k h_{j,ki}h_{k,ij}+2\la_k\la_i h_{k,ij}h_{i,jk}
          \endaligned
\end{equation}
and
\begin{equation}
IV_i=(1+\la_i^2)h_{i,ii}^2+\sum_{1\leq j\leq p,j\neq i}\big[(2+\la_j^2)h_{j,ij}^2+h_{i,jj}^2+2\la_i\la_j h_{i,jj}h_{j,ij}\big]+
C_1\big(\sum_j \la_j h_{j,ij}\big)^2.
\end{equation}
Here we  group the terms  according to different types of the
indices of the coefficient of the second fundamental form similarly
to \cite{J-X-Y}. Note that both $I_j$ and $II_{ijk}$ vanish whenever
$p=n$, and $III_{ijk}$ vanishes whenever $p\leq 2$.

Obviously
\begin{equation}\label{I}
I_i\geq 2\sum_{1\leq j\leq p}h_{j,ij}^2.
\end{equation}
As shown in (3.16) of \cite{J-X-Y},
\begin{equation}\label{II}
II_{ijk}\geq (3-v)(h_{k,ij}^2+h_{j,ik}^2).
\end{equation}
 Using Cauchy-Schwarz inequality, one can proceed as in Lemma 3.1 of \cite{J-X-Y} to get
\begin{equation}\label{III}
III_{ijk}\geq (3-v)(h_{i,jk}^2+h_{j,ki}^2+h_{k,ij}^2).
\end{equation}

Denote
\begin{equation}
\tau:=\f{1}{2}(v-1),
\end{equation}
then
\begin{equation}\label{IV1}
\aligned
&IV_i-\f{1}{2}(3-v)\big(h_{i,ii}^2+\sum_{1\leq j\leq p,j\neq i}(h_{i,jj}^2+2h_{j,ij}^2)\big)\\
=&\sum_{1\leq j\leq p,j\neq i}\Big[(2\tau+\la_j^2)h_{j,ij}^2+\tau h_{i,jj}^2+2\la_i\la_j h_{i,jj}h_{j,ij}\Big]\\
 &+(\tau+\la_i^2)h_{i,ii}^2+C_1\big(\sum_j \la_j h_{j,ij}\big)^2.
\endaligned
\end{equation}
Completing the square yields
\begin{equation}\aligned
&(2\tau+\la_j^2)h_{j,ij}^2+\tau h_{i,jj}^2+2\la_i\la_j h_{i,jj}h_{j,ij}\\
=&(\tau^{\f{1}{2}}h_{i,jj}+\tau^{-\f{1}{2}}\la_i\la_j h_{j,ij})^2+(2\tau+\la_j^2-\tau^{-1}\la_i^2\la_j^2)h_{j,ij}^2\\
\geq&(2\tau+\la_j^2-\tau^{-1}\la_i^2\la_j^2)h_{j,ij}^2.
\endaligned
\end{equation}
Substituting it into (\ref{IV1}) implies
\begin{equation}\aligned\label{IV2}
&IV_i-\f{1}{2}(3-v)\big(h_{i,ii}^2+\sum_{1\leq j\leq p,j\neq i}(h_{i,jj}^2+2h_{j,ij}^2)\big)\\
\geq& (\tau+\la_i^2)h_{i,ii}^2+\sum_{1\leq j\leq p,j\neq i}(2\tau+\la_j^2-\tau^{-1}\la_i^2\la_j^2)h_{j,ij}^2+C_1\big(\sum_{1\leq j\leq p}\la_j h_{j,ij}\big)^2.
\endaligned
\end{equation}

If there exists 2 distinct indices $j,k\neq i$ satisfying
$$2\tau+\la_j^2-\tau^{-1}\la_i^2\la_j^2\leq 0$$
and
$$2\tau+\la_k^2-\tau^{-1}\la_i^2\la_k^2\leq 0,$$
then $\la_i^2>\tau$ and
$$\la_j^2,\la_k^2\geq \f{2\tau^2}{\la_i^2-\tau}.$$
It implies
$$\aligned
(1+\la_i^2)(1+\la_j^2)(1+\la_k^2)&\geq \f{(\la_i^2+1)(\la_i^2+2\tau^2-\tau)^2}{(\la_i^2-\tau)^2}\geq \f{(2\tau+1)^3}{\tau+1}\\
                                 &=\f{2v^3}{v+1}>v^2
\endaligned$$
where the second equality holds if and only if
$\la_i^2=\tau(2\tau+3)=\f{1}{2}(v-1)(v+2)$ (see (3.25) in
\cite{J-X-Y}). We obtain a contradiction to (\ref{v1}). Hence one
can find an index $k\neq i$, such that
$$2\tau+\la_j^2-\tau^{-1}\la_i^2\la_j^2>0\qquad \text{for all }j\notin\{i,k\}.$$
Denote
$$s:=\sum_{1\leq j\leq p,j\neq k} \la_j h_{j,ij},$$
then by Cauchy-Schwarz inequality,
\begin{equation}\label{IV3}
\aligned
&(\tau+\la_i^2)h_{i,ii}^2+\sum_{1\leq j\leq p,j\notin\{i,k\}}(2\tau+\la_j^2-\tau^{-1}\la_i^2\la_j^2)h_{j,ij}^2\\
\geq & \Big(\f{\la_i^2}{\tau+\la_i^2}+\sum_{1\leq j\leq p,j\notin\{i,k\}}\f{\la_j^2}{2\tau+\la_j^2-\tau^{-1}\la_i^2\la_j^2}\Big)^{-1}s^2.
\endaligned
\end{equation}
Now we denote
$$\aligned
a&:=\f{\la_i^2}{\tau+\la_i^2}+\sum_{1\leq j\leq p,j\notin\{i,k\}}\f{\la_j^2}{2\tau+\la_j^2-\tau^{-1}\la_i^2\la_j^2}, \\
b&:=2\tau+\la_k^2-\tau^{-1}\la_i^2\la_k^2.
\endaligned$$
We only need to deal with $b<0$ case. Substituting (\ref{IV3}) into
(\ref{IV2}) gives
\begin{equation}\label{IV4}\aligned
&IV_i-\f{1}{2}(3-v)\big(h_{i,ii}^2+\sum_{1\leq j\leq p,j\neq i}(h_{i,jj}^2+2h_{j,ij}^2)\big)\\
\geq& a^{-1}s^2+bh_{k,ik}^2+C_1(s+\la_k h_{k,ik})^2\\
=&(C_1+a^{-1})s^2+(C_1\la_k^2+b)h_{k,ik}^2+2C_1\la_k sh_{k,ik}
\endaligned
\end{equation}
which is nonnegative  if and only if
$$0\leq (C_1+a^{-1})(C_1\la_k^2+b)-(C_1\la_k)^2=C_1(a^{-1}\la_k^2+b)+a^{-1}b.$$
Hence the remain work is to prove
$$\f{a^{-1}\la_k^2+b}{a^{-1}b}=a+b^{-1}\la_k^2=\f{\la_i^2}{\tau+\la_i^2}+\sum_{1\leq j\leq p,j\neq i}\f{\la_j^2}{2\tau+\la_j^2-\tau^{-1}\la_i^2\la_j^2}$$
is bounded from above by $-\de_0$ with a positive constant $\de_0$;
thereby $C_1:=\de_0^{-1}$ is the required constant.

Without loss of generality, we assume $i=1$ and $k=2$; let $r=\la_1^2$ and $t\in \R^+$ satisfying
$$(1+r)(1+t)=v^2.$$
As shown in the proof of Lemma 3.2 in \cite{J-X-Y},
$$\f{\la_1^2}{\tau+\la_1^2}+\sum_{2\leq j\leq p}\f{\la_j^2}{2\tau+\la_j^2-\tau^{-1}\la_1^2\la_j^2}\leq \f{r}{\tau+r}+\f{t}{2\tau+t-\tau^{-1}rt}.$$
Hence it suffices to prove
\begin{equation}\label{F0}
\sup_\Om F(r,t)\leq -\de_0
\end{equation}
for a positive constant $\de_0$ not depending on $v\in (1,3)$,
where
\begin{equation}
F(r,t):=\f{r}{\tau+r}+\f{t}{2\tau+t-\tau^{-1}rt}
\end{equation}
and
\begin{equation}\label{Om}
\Om:=\{(r,t)\in \R^+\times \R^+: (1+r)(1+t)=v^2, r>\tau, t\geq \f{2\tau}{\tau^{-1}r-1}\}.
\end{equation}

We rewrite $F(r,t)$ as
\begin{equation}\label{F}\aligned
F(r,t)&=\big(1+\f{\tau}{r}\big)^{-1}+\big(1-\tau^{-1}r+\f{2\tau}{t}\big)^{-1}\\
      &=\big(1+\f{\tau}{r}\big)^{-1}\big(2+\tau\big(\f{1}{r}+\f{2}{t}\big)-\tau^{-1}r\big)\big(1-\tau^{-1}r+\f{2\tau}{t}\big)^{-1}\\
      &:=F_1(r)^{-1}F_2(r,t)\big(F_2(r,t)-F_1(r)\big)^{-1}
      \endaligned
\end{equation}
From $r\geq \tau$ we immediately get
\begin{equation}\label{F1}
F_1(r)=1+\f{\tau}{r}\leq 2.
\end{equation}
$(1+r)(1+t)=v^2$ gives $t=\f{v^2-1-r}{1+r}$ and moreover
\begin{equation}\aligned
F_2(r,t)&=2+\tau\big(\f{1}{r}+\f{2}{t}\big)-\tau^{-1}r=2+\tau\big(\f{1}{r}+\f{2(1+r)}{v^2-1-r}\big)-\tau^{-1}r\\
        &=\f{r^3+(2\tau^2-2\tau-v^2+1)r^2+(\tau^2+2\tau(v^2-1))r+\tau^2(v^2-1)}{\tau\, r(v^2-1-r)}\\
        &=\f{2r^3-(v-1)(v+5)r^2+(v-1)^2(2v+\f{5}{2})r+\f{1}{2}(v-1)^3(v+1)}{(v-1) r(v^2-1-r)}.
\endaligned
\end{equation}
Let $\th:=\f{r}{v-1}$, then $\tau< r\leq (1+r)(1+t)-1=v^2-1$ implies $\th\in (\f{1}{2},v+1]$, and
\begin{equation}\label{F2}
F_2=\f{2\th^3-(v+5)\th^2+(2v+\f{5}{2})\th+\f{1}{2}(v+1)}{\th(v+1-\th)}:=\f{H_1(v,\th)}{H_2(v,\th)}.
\end{equation}
It is easily-seen that
\begin{equation}\label{H2}
H_2(v,\th)=\th(v+1-\th)\leq \f{1}{4}(v+1)^2\leq 4.
\end{equation}
Denote
$$\mathcal{D}:=\big\{(v,\th):v\in (1,3), \th\in (\f{1}{2},v+1]\big\}.$$
Observing that $H_1(\cdot,\th)$ is an affine linear function in $v$
for any fixed $\th$, we have
\begin{equation*}
\inf_\mathcal{D}H_1=\min\Big\{\inf_{\th\in [\f{1}{2},2]}H_1(1,\th),\inf_{\th\in [2,4]} H_1(\th-1,\th),\inf_{\th\in [\f{1}{2},4]}H_1(3,\th)\Big\}.
\end{equation*}
A straightforward calculation shows
$$\aligned
\inf_{\th\in [\f{1}{2},2]}H_1(1,\th)&=\inf_{\th\in [\f{1}{2},2]}\big(2\th(\th-\f{3}{2})^2+1\big)=1,\\
\inf_{\th\in [2,4]}H_1(\th-1,\th)&=\inf_{\th\in [2,4]}\th(\th-1)^2=2,\\
\inf_{\th\in [\f{1}{2},4]}H_1(3,\th)&=\inf_{\th\in [\f{1}{2},4]}\big(2\th(\th-2)^2+\f{1}{2}\th+2\big)>2.
\endaligned$$
Therefore
\begin{equation}\label{H1}
H_1(v,\th)\geq 1.
\end{equation}
Substituting (\ref{H2}) and (\ref{H1}) into (\ref{F2}) gives
\begin{equation}
F_2\geq \f{1}{4}.
\end{equation}
Since $F_2-F_1=1-\tau^{-1}r+\f{2\tau}{t}<0$, we have
\begin{equation}\label{F12}
|F_2-F_1|\leq |-F_1|\leq 2.
\end{equation}
Thereby one can obtain
\begin{equation}
|F|=|F_1|^{-1}F_2|F_2-F_1|^{-1}\geq \f{1}{16}
\end{equation}
by substituting (\ref{F1}), (\ref{F2}) and (\ref{F12}) into
(\ref{F}); in other words, (\ref{F0}) holds true with positive
constant $\de_0=\f{1}{16}$ and moreover the right hand side of
(\ref{IV4}) is non-negative, i.e.
\begin{equation}\label{IV}
IV_i\geq \f{1}{2}(3-v)\big(h_{i,ii}^2+\sum_{1\leq j\leq p,j\neq i}(h_{i,jj}^2+2h_{j,ij}^2)\big).
\end{equation}

Finally (\ref{La4}) is followed from (\ref{group}), (\ref{I}), (\ref{II}), (\ref{III}) and (\ref{IV}).
\end{proof}

\begin{rem}

From the above proof, one can take $C_1=16$ in (\ref{La4}).

\end{rem}

Let $h:=h(\log v)$, then
$$\aligned
\mc{L}h&=\De h-\f{1}{2}\lan X,\n h\ran=h''|\n \log v|^2+h'\De \log v-\f{1}{2}h'\lan X,\n \log v\ran\\
&=h'\mc{L}(\log v)+h''|\n \log v|^2
\endaligned$$
where $h'$ ($h''$) is the first (second) derivation of $h$ with
respect to $\log v$. Now we choose
\begin{equation}
h:=v^{C_1}=\exp(C_1\log v),
\end{equation}
then by Proposition \ref{subhar3},
\begin{equation}\label{L3}
\mc{L}h=C_1 h\big(\mc{L}(\log v)+C_1|\n \log v|^2\big)\geq \f{1}{2}C_1h(3-v)|B|^2.
\end{equation}
Then one can proceed as in \S \ref{hyp} to obtain

\begin{thm}\label{Ri3}
Let $u=(u^1,\cdots,u^m):\R^n\ra \R^m$ be a smooth vector-valued function, such that $M=\text{graph}\ u$ is a self-shrinker. If the slope of $u$
\begin{equation}\label{slope}
\De_u:=\det\Big(\de_{ij}+\sum_\a u_i^\a u_j^\a\Big)^{\f{1}{2}}<3,
\end{equation}
then $u$ has to be linear and $M$ has to be a linear subspace.
\end{thm}

\begin{proof}
Denote $F:\R^n\ra M$
$$x=(x_1,\cdots,x_n)\mapsto \big(x,u(x)\big)$$
then $F$ is a diffeomorphism and $M$ can be viewed as $\R^n$
equipped with metric $g=g_{ij}dx_idx_j$, where
$g_{ij}=\de_{ij}+u_i^\a u_j^\a.$ Obviously the eigenvalues of
$(g_{ij})$ are bounded from below by 1. Moreover
$\De_u=\det(g_{ij})^{\f{1}{2}}<3$ implies
\begin{equation}
|\xi|^2\leq \xi_ig_{ij}\xi_j<9|\xi|^2\qquad \forall \xi\in \R^n.
\end{equation}
Hence $M$ is a simple Riemannian manifold with Euclidean volume growth.

(\ref{slope}) equals to say $v=\De_u$ is a smooth
function on $M$ taking values in $[1,3)$. Let $\phi$ be a smooth
function on $M$ with compact supporting set, multiplying $\phi^2 h$
with both sides of (\ref{L3}) and integrating by parts yield
$$\aligned
\f{1}{2}C_1\int_M \phi^2 h^2(3-v)|B|^2\rho&\leq \int_M \phi^2 h\ \text{div}(\rho \n h)\\
&=-\int_M \lan \n(\phi^2 h),\n h\ran \rho\\
&=-\int_M \phi^2|\n h|^2\rho-2\int_M \lan h\n \phi,\phi\n h\ran \rho\\
&\leq -\f{1}{2}\int_M \phi^2|\n h|^2\rho+2\int_M |\n \phi|^2
h^2\rho,
\endaligned$$
i.e.
\begin{equation}\label{L4}
C_1\int_M \phi^2 h^2(3-v)|B|^2\rho+\int_M \phi^2|\n h|^2\rho\leq 4\int_M |\n \phi|^2 h^2\rho.
\end{equation}
Now we choose $\phi$ to be a cut-off function which satisfies $\phi\equiv 1$ on $D_R$, $\phi\equiv 0$ outside $D_{2R}$ and $|\n\phi|\leq \f{c_0}{R}$,
then
\begin{equation}\aligned
\int_{D_R}|\n h|^2\rho&\leq \int_M \phi^2|\n h|^2 \rho\leq 4\int_M |\n\phi|^2 h^2\rho\\
                      &=4\int_{D_{2R}\backslash D_R}|\n \phi|^2 h^2\rho\leq \f{4c_0^2}{R^2}3^{2C_1}e^{-\f{R^2}{4}}\text{Vol}(D_{2R}\backslash D_R).
                      \endaligned
\end{equation}
Letting $R\ra +\infty$ forces $|\n h|^2\equiv 0$, furthermore $v\equiv \text{const}$. Assume $v\equiv v_0<3$ and denote
$h_0:=v_0^{C_1}$. Again using (\ref{L4}) gives
\begin{equation}\aligned
C_1h_0^2(3-v_0)\int_{D_R}|B|^2\rho&\leq C_1\int_M \phi^2 h^2(3-v)|B|^2\rho\\
 &\leq 4\int_M |\n \phi|^2h^2\rho\\
 &\leq \f{4c_0^2}{R^2}3^{2C_1}e^{-\f{R^2}{4}}\text{Vol}(D_{2R}\backslash D_R).
 \endaligned
 \end{equation}
Let $R\ra +\infty$, we have $|B|^2\equiv 0$. Hence $M$ has to be a linear subspace.

\end{proof}

\bigskip

\Section{Appendix}{Appendix}

\medskip

Let $f:\ (M^m,g)\rightarrow (N^n,h)$ be a smooth map and let $w$ be a given smooth positive function in $M$.
Let $\{e_i\}_{i=1}^m$ be a local orthonormal frame field in $M$, $\n$ be the Levi-Civita connection of $M$ and $\lan\cdot,\cdot\ran$ be the inner
product of $(N,h)$.
Then we can define $w-$weighted energy density of the given map $f$ by
$$e_w(f)=\f12\lan f_*e_i,f_*e_i\ran w.$$
Let $\Om$ be an arbitrary domain in $M$ such that $\overline{\Om}$ is compact. The integral of the $w-$weighted energy density over $\Om$ yields the $w-$weighted energy of the map $f$:
$$E_{w,\Om}(f)=\int_\Om e_w(f)=\f12\int_\Om \lan f_*e_i,f_*e_i\ran w.$$

Now let's deduce the first variational formula for weighted map.

Let $f_t$ be a 1-parameter family of maps from $M$ to $N$ with $f_0=f$ and $\f{df_t}{dt}\mid_{t=0}$ having compact support in $\Om$. Let $\{e_i\}$ be a normal coordinate frame of $M$ at $p\in M$, then by (1.2.12) of \cite{X3},
\begin{equation}\aligned\label{1var}
\f{d}{dt}E_{w,\Om}(f_t)=&\f12\f{d}{dt}\int_\Om\lan f_{t*}e_i,f_{t*}e_i\ran w\\
=&\int_\Om\left\lan\n_{\f{\p}{\p t}} f_{t*}e_i,f_{t*}e_i\right\ran w
=\int_\Om\left\lan\n_{e_i}\f{df_t}{dt},f_{t*}e_i\right\ran w\\
=&\int_\Om\left\lan \n_{e_i}\left(\left\lan\f{df_t}{dt},wf_{t*}e_j\right\ran e_j\right),e_i\right\ran-\int_\Om\left\lan\f{df_t}{dt},\n_{e_i}(wf_{t*}e_i)\right\ran\\
=&\int_\Om\mathrm{div}\left(\left\lan\f{df_t}{dt},wf_{t*}e_j\right\ran e_j\right)-\int_\Om\left\lan\f{df_t}{dt},\n_{e_i}(wf_{t*}e_i)\right\ran\\
=&-\int_\Om\left\lan\f{df_t}{dt},\tau_w(f_t)\right\ran w.\\
\endaligned
\end{equation}
Here, $\tau_w(f)=w^{-1}\n_{e_i}(wf_{*}e_i)=w^{-1}(wf^\a_i)_i\f{\p}{\p y_\a}$, and $df=\f{\p f^\a}{\p x_i}dx_i\otimes\f{\p}{\p y_\a}$ is a section of the bundle $T^*M\otimes f^{-1}TN$.

We call $f$ is a \emph{$w-$weighted harmonic map} if $\tau_w(f)=0$.

\bibliographystyle{amsplain}

\end{document}